\documentclass[a4paper,12pt,onecolumn]{article}

\usepackage{hyperref}
\usepackage[vmargin=2cm,hmargin=2cm,headheight=14.5pt,top=2cm,headsep=.5cm]{geometry}

\usepackage{bm}
\usepackage[utf8]{inputenc}

\usepackage[autobold]{mathfixs}
\usepackage{empheq}
\usepackage{stackrel}
\usepackage{cases}
\usepackage{float}
\usepackage{mathtools}
\usepackage{amsthm,amsmath,amscd}
\usepackage{makeidx}
\usepackage[charter]{mathdesign}
\usepackage{perpage}
\usepackage{bm}
\usepackage{pgfplots}
\pgfplotsset{compat=1.15}
\usepackage{mathrsfs}
\usetikzlibrary{arrows}
\usepackage{tikz-cd}
\tikzcdset{every label/.append style = {font = \small}}

\usepackage{caption}
\usepackage{xcolor}
\usepackage{graphicx}
\DeclareMathSizes{12}{12}{8}{6}

\usepackage{cite}
\usepackage{url}

\usepackage[ddmmyy]{datetime}
\usepackage{enumitem}
\setlist[itemize,2]{label=$\centerdot$}
\setlist[itemize,3]{label=$\triangle$}
\usepackage{framed}
\usepackage[symbol]{footmisc}
\usepackage{marvosym}
\usepackage{perpage}
\MakePerPage[1]{footnote}

\makeatletter
\renewcommand*{\@makefnmark}{\hbox{\@textsuperscript{%
			\normalfont\@thefnmark}}}
\makeatother
\makeatletter
\def\@fnsymbol#1{\ensuremath{\ifcase#1\or \text{\Mercury}
		\or \text{\Venus} \or \text{\Earth} \or
		\text{\Jupiter} \or \text{\Saturn} \or \text{\Neptune} \or \text{\Uranus} \or
		\text{\Pluto}
		\or \text{\Moon} \or \text{\Sun}
		\else\@ctrerr\fi}}%
\makeatother

\newtheoremstyle{ptheorem}{1em}{0em}{\itshape}{}{\bfseries}{.}{.5em}{\thmname{#1}\thmnumber{
		#2}\thmnote{ (\hspace{-.01pt}{#3})}}

\theoremstyle{ptheorem}

\newtheorem{thm}{Theorem}[section]

\newtheorem{lem}[thm]{Lemma}
\newtheorem{cor}[thm]{Corollary}

\newtheoremstyle{hdef}{1em}{0em}{}{}{\bfseries}{.}{.5em}{\thmname{#1}\thmnumber{
		#2}\thmnote{ (\hspace{-.01pt}{#3})}}
\theoremstyle{hdef}

\newtheorem{dfn}[thm]{Definition}
\newtheorem{rem}[thm]{Remark}

\newtheorem{exa}[thm]{Example}

\numberwithin{equation}{section}
\numberwithin{figure}{section}

\renewcommand{\phi}{\varphi}

\newcommand{\olb}[1]{%
	\vbox{\offinterlineskip\ialign{\hfil##\hfil\cr
			$\rotatebox[origin=c]{90}{$]$}$\cr\noalign{\kern-.45ex}{$#1$}\cr}}}

\newcommand{\noop}[1]{}

\usepackage{stmaryrd}

\allowdisplaybreaks

\begin{document}

	\title{Necessary and sufficient conditions for distances\\on the real line}

\date{}

\author{Daniel Cao Labora$^{*\dagger}$\\
	\small e-mail:daniel.cao@usc.es\\
	Francisco J. Fernández$^*$\\
	\small e-mail: fjavier.fernandez@usc.es\\
	F. Adri\'an F. Tojo$^*$ \\
	\small e-mail: fernandoadrian.fernandez@usc.es
	\\Carlos Villanueva\\ \small e-mail: carlos.villanuevamariz@lmh.ox.ac.uk}

\maketitle

\begin{abstract}
When dealing with certain mathematical problems, it is sometimes necessary to show that some function induces a metric on a certain space. When this function is not a well renowned example of a distance, one has to develop very particular arguments that appeal to the concrete expression of the function in order to do so. The main purpose of this paper is to provide several sufficient results ensuring that a function of two variables induces a distance on the real line, as well as some necessary conditions, together with several examples that show the applicability of these results. In particular, we show how a hypothesis about the sign of the cross partial derivative of the candidate to distance is helpful for deriving such kind of results.
\end{abstract}


{\small\textbf{MSC 2020:} 26B99, 51N20, 54E35, 00A08

\footnotetext{\noindent$^*$\emph{Departamento de Estatística, Análise Matemática e Optimización, Universidade de Santiago de Compostela, 15782, Facultade de Matemáticas, Santiago, Spain.}}
\footnotetext{\noindent$^\dagger$\emph{Corresponding author: daniel.cao@usc.es}}

\maketitle

\section{Motivation and introduction}

Throughout the rest of the document we will focus on distances on the set of real numbers $\mathbb{R}$. Thus, it is suitable to recall the definition of distance in the particular case of a distance on $\mathbb{R}$.

\begin{dfn} Given $d: \mathbb{R}^2 \to \mathbb{R}$, we say that $d$ is a \emph{metric} or \emph{distance} whenever it fulfills the following three properties simultaneously:
\begin{itemize}
\item \emph{Positive definiteness:} $d(x,y) \geq 0$ for any $x,y \in \mathbb{R}$, where $d(x,y)=0$ if and only if $x=y$.
\item \emph{Symmetry:} $d(x,y)=d(y,x)$ for any $x,y \in \mathbb{R}$.
\item \emph{Triangle inequality:} $d(x,y)+d(y,z) \leq d(x,z)$ for any $x,y,z \in \mathbb{R}$.
\end{itemize}
\end{dfn}

If one is given a certain function $d: \mathbb{R}^2 \to \mathbb{R}$ and is asked to prove that it is a distance on the real line, it is quite reasonable to proceed as follows. First, the symmetry of the $d$ should be quite clear, just by inspecting that $d$ stays invariant when interchanging the roles of $x$ and $y$. Positive definiteness should also be direct, or sometimes a mere consequence of a tricky factorization of $d$ that shows that the function is a square that only vanishes for $x=y$. Regarding the triangle inequality, one could try to use the very particular expression of $d$ in order to prove it, or some arguments involving concavity/convexity. However, how to state reasonably general theorems, with easy-to-check hypotheses and that ensure the triangle inequality is fulfilled does not seem immediate. This quest guides the main topic of this paper, together with the description of some necessary conditions for $d$ being a metric and a special mention to the case of translation invariant distances.

\section{A special case: translation invariant distances}

If we are interested in metrics on the real line, it is quite reasonable to put our initial goal on translation invariant distances. Informally, we consider distances such that, rather depending on the two variables $x$ and $y$, they only depend on the difference $x-y$. Thus, $d:{\mathbb R}^2\to \mathbb{R}$ is said to be a \emph{translation invariant distance} if $d$ is a distance and $d(x+z,y+z)=d(x,y)$ for every $x,y,z\in \mathbb{R}$.

In this very particular case, it is not complicated to characterize such distances. The fundamental notion necessary is that of subadditive function. In this sense, we say that $f:\mathbb{R} \to \mathbb{R}$ is \emph{subadditive} if $f(x+y)\leq f(x)+f(y)$ for any pair $x,y \in \mathbb{R}$. 

\begin{thm} \label{t0} The function $d:{\mathbb R}^2\to{\mathbb R}$ is a translation invariant distance if and only if it is of the form $d(x,y)=f(y-x)$ where $f:\mathbb{R} \to \mathbb{R}$ is an even subadditive function with $f(0)=0$ and $f(x)>0$ for any $x \neq 0$.
\end{thm}

\begin{proof}First assume that $d:{\mathbb R}^2\to{\mathbb R}$ is a translation invariant distance. Define $f(x)=d(0,x)$ for $x\in \mathbb{R}$. Clearly $f(0)=0$, $f(x)>0$ for any $x \neq 0$, and $d(x,y)=d(0,y-x)=f(y-x)$ for any $x,y\in \mathbb{R}$, since $d$ is translation invariant. Besides, $f$ is even since, for $x \in \mathbb{R}$, \[ f(x)=d(0,x)=d(x,0)=d(0,-x)=f(-x).\]  Furthermore, for any $x,y\in \mathbb{R}$, \[ f(x+y)=d(0,x+y)=d(-x,y)\leq d(-x,0)+d(0,y)=d(0,x)+d(0,y)=f(x)+f(y),\]  so $f$ is subadditive. 

On the other hand, if $f:\mathbb{R}\to \mathbb{R}$ is an even subadditive function with $f(0)=0$ and $f(x)>0$ for $x \neq 0$, let us define $d(x,y):=f(y-x)$ and prove that $d$ is a distance. Indeed, $d(x,x)=f(0)=0$ for every $x\in {\mathbb R}$, $d(x,y)=f(y-x)>0$ for every $x \neq y$, and $d(x,y)=f(y-x)=f(x-y)=d(y,x)$ for every $x,y\in{\mathbb R}$. Finally, \[ d(x,z)=f(z-x) = f(z-y+y-x)\leq f(z-y)+f(y-x)=d(y,z)+d(x,y),\]  for every $x,y,z\in{\mathbb R}$, so the triangle inequality holds.
\end{proof}

\begin{rem}Observe that, as a consequence of Theorem~\ref{t0}, a distance $d$ is translation invariant if and only if it can be factorized as $d = g \circ l$, where $l: \mathbb{R}^2 \to [0,\infty)$ is the usual distance, $l(x,y)=\vert x - y \vert$, and $g: [0,\infty) \to [0,\infty)$ is a function with $g(0)=0$, $g(x)>0$ for $x>0$, and whose even extension is subadditive.
\end{rem}
The following example shows that, in general, the even extension of a subadditive function is not subadditive. This implies that, in the previous paragraph, it is not enough to ensure that $g$ is subadditive, but we need to ensure that the even extension of $g$ is subadditive.

\begin{exa}\label{exa1} Consider the even function $f:\mathbb{R} \to \mathbb{R}$ such that 
\[ f(x)=	\begin{dcases}
			\left\lvert x\right\rvert, & 0\leq \vert x \vert <1, \\
			2-\left\lvert x\right\rvert, & 1\leq \vert x \vert <\frac{5}{3}, \\
			\frac{1}{3}, & \frac{5}{3}\leq \vert x \vert. \end{dcases}\] 
\begin{figure}[h!]
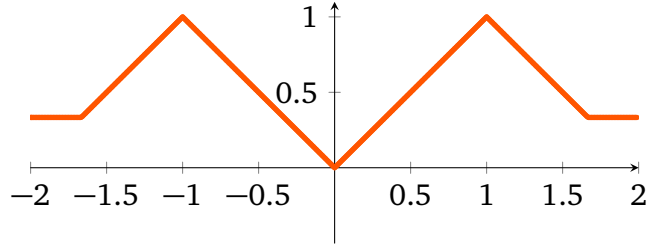

	\centering
\definecolor{ffvvqq}{rgb}{1.,0.3333333333333333,0.}

	\caption{Graph of the function $f$ in Example~\ref{exa1}.}
\end{figure}			
Clearly, $f$ is not subadditive since $f(3-2)=f(1)=1>2/3=f(3)+f(-2)$. Nevertheless, the restriction of $f$ to $[0,\infty)$, which we will denote by $g$, is subadditive. In order to prove this claim, let us consider $x \geq y \geq 0$ and analyze a few cases.
\begin{enumerate}
\item If $x+y \geq 2$, then $x \geq 1$. Thus, $g(x+y) \leq g(x)+g(y)$ since $g$ is decreasing on $[1,\infty)$.
\item If $1 \leq x+y \leq 2$, the subadditivity is clear when $x \geq 1$ since $g$ is decreasing on $[1,\infty)$. If $x<1$, then $ x \geq y>0$ and $g(x)+g(y)=x+y \geq 1 \geq g(x+y)$.
\item If $x+y \leq 1$, then simply $g(x+y)=x+y=g(x)+g(y)$.
\end{enumerate}
Thus, in this case, the definition $d(x,y)=g(\vert x-y \vert)$ does not induce a distance, even though $g$ is subadditive, and the underlying reason is that its even extension $f$ is not subadditive.
\end{exa}

\begin{rem}We observe that, when $g$ is non-decreasing and subadditive, then its even extension $f$ is automatically subadditive. The key consideration for proving this relies on the fact that, for any $x,y \in \mathbb{R}$, we have $g(\vert x + y \vert) \leq g(\vert x\vert)+ g(\vert y \vert)$. Observe how if $x$ and $y$ have the same sign, this is simply the subadditivity of $g$ because $\vert x + y \vert = \vert x \vert + \vert y \vert$. If $x$ and $y$ have different signs, then either $\vert x + y \vert < \vert x \vert$ or $ \vert x + y \vert < \vert y \vert$ and, since $g$ is non-decreasing, $g(\vert x + y \vert) \leq g(\vert x\vert)+ g(\vert y \vert)$.
\end{rem}
\begin{rem} \label{r1}
If $g:[0,\infty) \to [0,\infty)$ is non-decreasing and subadditive and $f$ is the even extension of $g$, we have that \[ f(x+y)=g(\vert x + y \vert) \leq g(\vert x\vert)+ g(\vert y \vert)=f(x)+f(y),\]  for every $x,y \in \mathbb{R}$. Thus, under these hypotheses, $f$ is also subadditive.
\end{rem}

Consequently, from Theorem~\ref{t0} and Remark~\ref{r1}, we derive the following corollary.

\begin{cor} \label{c1} The function $d:{\mathbb R}^2\to{\mathbb R}$ is a translation invariant distance fulfilling $d(0,x) \leq d(0,y)$ for any $x,y \in \mathbb{R}$, where $0 \leq x \leq y$, if and only if it is of the form $d(x,y)=g(\vert y-x \vert)$ where $g:[0,\infty) \to \mathbb{R}$ is a subadditive non-decreasing function with $g(0)=0$ and $g(x)>0$ for any $x > 0$.
\end{cor}

Finally, we observe that there are examples of distances which can be obtained from the even extension $f$ of a non-monotonic function $g$ where $f$ is subadditive. We show this via the following suitably modified version of the previous example.
\begin{exa}\label{exa2} Consider the even function $f:\mathbb{R} \to \mathbb{R}$ fulfilling 
\[ f(x)=	\begin{dcases}
			\left\lvert x\right\rvert, & 0\leq \vert x \vert <1, \\
			2-\left\lvert x\right\rvert, & 1\leq \vert x \vert <\frac{4}{3}, \\
			\frac{2}{3}, & \frac{4}{3}\leq \vert x \vert. \end{dcases}\] 

\begin{figure}[h!]
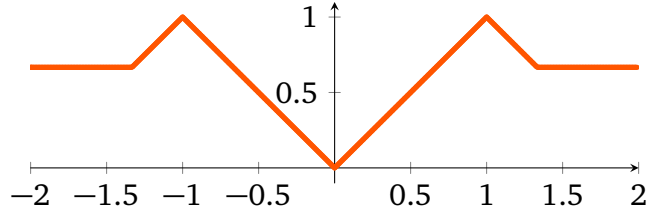

	\centering

	\caption{Graph of the function $f$ in Example~\ref{exa2}.}
\end{figure}

We claim that $f$ is subadditive, and we will prove this by distinguishing several cases. Note that, since $f$ is even, we can assume $x+y \geq 0$. Besides, without loss of generality, we will assume $x \geq y$, so we will also have $x \geq 0$.
\begin{enumerate}
\item If $x+y \geq 2$, then $x \geq 1$. Thus, $f(x+y) \leq f(x)+f(y)$ since $f$ is decreasing on $[1,\infty)$ and $f(x+y)=\min\{f(s):\; s\in [1,\infty)\}$.
\item If $1 \leq x+y \leq 2$, the subadditivity is clear when $x \geq 1$ since $f$ is decreasing on $[1,\infty)$. If $x<1$, then $1>x \geq y>0$ necessarily, so $f(x)+f(y)=x+y \geq 1 \geq f(x+y)$.
\item If $2/3 \leq x+y \leq 1$, and $x \leq x+y$ then $x \geq y \geq 0$ necessarily, so $f(x)+f(y)=x+y =f(x+y)$. The complicated case happens when $x \geq x+y$, so $y \leq 0$. If $x \leq 1$, then $f(x)+f(y)=x-y \geq x+y=f(x+y)$. If $1 \leq x \leq 4/3$, then $f(x)+f(y)=2-(x+y) \geq x+y$. If $x \geq 4/3$, then $y \leq -1/3$ so $f(x)+f(y) \geq 1 \geq f(x+y)$.
\item If $0 \leq x+y \leq 2/3$, the result is clear when $x \geq x+y$ since $f(x+y) = x+y$ and the minimum for $f$ on $[x+y,\infty)$ is $x+y$. If $x \leq x+y$, then $x \geq y \geq 0$ necessarily, so $f(x)+f(y)=x+y =f(x+y)$.
\end{enumerate}
\end{exa}

%

\section{Necessary conditions}

In the previous section we have dealt with the particular case of translation invariant metrics on $\mathbb{R}$, providing a characterization for such distances. Now, we focus on the generic case of a distance on the real line, that will be denoted by $d: \mathbb{R}^2 \to \mathbb{R}$. In this case, we cannot expect to find a characterization in order to know whether $d$ is a distance or not, but only necessary or sufficient conditions.

In this section we prove some necessary conditions on $d$, provided that $d: \mathbb{R}^2 \to \mathbb{R}$ is a metric. In the rest of the document, $\Delta \subset \mathbb{R}^2$ will denote the diagonal of the cartesian plane, that is, \[ \Delta=\{(x,x) \in \mathbb{R}^2: x \in \mathbb{R}\}.\]  Let $X:=\{(x,y)\in{\mathbb R}^2\ :\ x\leq y\}$, $Y:=\{(x,y)\in{\mathbb R}^2\ :\ y\leq x\}$. Given a function $d:{\mathbb R}^2\to{\mathbb R}$ and $(x,y),v\in{\mathbb R}^2$, we will define the \emph{directional derivative from the right} as \[ \partial_v^+ d(x,y):=\lim_{h\to 0^+}\frac{d((x,y)+hv)-d(x,y)}{h},\]  in case the limit exists. We will use the notation $\partial_1^+$, $\partial_2^+$, $\partial_1^-$ and $\partial_2^-$ for the cases $v=(1,0)$, $(0,1)$, $(-1,0)$ and $(0,-1)$ respectively.

Since $X$ and $Y$ are not open sets, we provide a short comment regarding the notion of differentiability for $d|_X$ and $d|_Y$. In the interior of $X$ (respectively $Y$) the notion of differentiability is well known. With respect to the points of the form $(x,x)\in \Delta$, we understand that $d|_X$ is differentiable at $(x,x)\in \Delta$ if there exists $w\in{\mathbb R}^2$ such that for every $(\widetilde x,\widetilde y)\in X$,
\[ d(\widetilde x,\widetilde y)=d(x,x)+w\cdot(\widetilde x-x,\widetilde y-y)+o(\left\lVert (\widetilde x-x,\widetilde y-y)\right\rVert),\] 
where $\cdot$ denotes the scalar product, and $o$ is used for the Landau notation. Since we can take $(\widetilde x,\widetilde y)$ such that $(\widetilde x-x,\widetilde y-y)=(0,1)$ or $(\widetilde x-x,\widetilde y-y)=(-1,0)$, the choice for $w$, if it exists, is unique. In case of existence of such a $w$, we say that $w$ is the \emph{derivative or gradient of $d|_X$ at $(x,x)$} and we write $\nabla d( x, y)=w$. A similar definition goes for $Y$.

Observe that, due to the symmetry property, if $d\in{\mathcal C}({\mathbb R}^2,[0,\infty))$ is a distance and $d|_X$ is differentiable, then $d|_Y$ is differentiable too. Hence, if $d\in{\mathcal C}({\mathbb R}^2,[0,\infty))$ is a distance and $d|_X$ is differentiable, $\partial_v^+d(x,y)$ exists for every $(x,y)$ and any $v\in{\mathbb R}^2$. Besides, due to the symmetry of $d$, $\partial_1^+ d(x,y)=\partial_2^+ d(y,x)$ and $\partial_1^- d(x,y)=\partial_2^- d(y,x)$.

\subsection{Conditions involving first order derivatives}

 In this section we will prove some necessary conditions involving first order derivatives to guarantee that a function $d$ is a distance.

\begin{thm}\label{thmn1}
	Let $d\in{\mathcal C}({\mathbb R}^2,[0,\infty))$ be a distance such that $d|_X$ is differentiable, $(x,y)\in {\mathbb R}^2$. Then we have that $\left\lvert \partial_2^-d(x,y)\right\rvert\leq \left\lvert \partial_2^- d(y,y)\right\rvert$ and $\left\lvert \partial_2^+d(x,y)\right\rvert\leq \left\lvert \partial_2^+ d(y,y)\right\rvert$ for any $x,y \in \mathbb{R}$.
\end{thm} 

\begin{proof} Let $x,y,z\in{\mathbb R}$. By the triangle inequality, $d(x,z)\leq d(x,y)+d(y,z)$ and $d(x,y)\leq d(x,z)+d(y,z)$, so $\left\lvert d(x,z)-d(x,y)\right\rvert\leq d(y,z)$. Now,
\[ \left\lvert \frac{d(x,z)-d(x,y)}{\vert z-y \vert}\right\rvert\leq \left\lvert \frac{d(y,z)}{\vert z-y \vert}\right\rvert =\left\lvert \frac{d(y,z)-d(y,y)}{\vert z-y \vert}\right\rvert. \] 
If $z \to y^-$, we deduce $\left\lvert \partial_2^-d(x,y)\right\rvert\leq \left\lvert \partial_2^- d(y,y)\right\rvert$ for any $x,y \in \mathbb{R}$. Analogously, if $z \to y^+$, we deduce the inequality $\left\lvert \partial_2^+d(x,y)\right\rvert\leq \left\lvert \partial_2^+ d(y,y)\right\rvert$.
\end{proof}


%

\begin{rem}
The symmetry of $d$ implies that $\left\lvert \partial_1^-d(y,x)\right\rvert\leq \left\lvert \partial_1^- d(y,y)\right\rvert$ and $\left\lvert \partial_1^+d(y,x)\right\rvert\leq \left\lvert \partial_1^+ d(y,y)\right\rvert$ for any $x,y \in \mathbb{R}$.
\end{rem}

\begin{cor}\label{cor1}
	Let $d\in{\mathcal C}({\mathbb R}^2,[0,\infty))$ be a distance such that $d|_X$ is differentiable. Then, $d$ is not differentiable at any point of $\Delta$.
\end{cor}
\begin{proof}
Since $d(x,y)=d(y,x)$, $\partial_1 d(x,y)=\partial_2 d(y,x)$ for $x\ne y$, and $\partial_1^+ d(x,x)=\partial_2^+ d(x,x)$. Assume $d$ is differentiable at $(y,y)\in\Delta$. Then, if $v=(1,1)$ and we compute the directional derivative of $d$ in the direction of $v$, we get that, since $d(x,x)=0$ for every $x\in{\mathbb R}$,
	\[ 0=\partial_v^+ d(y,y)=\partial_1^+ d(y,y)+\partial_2^+ d(y,y)=2\partial_2^+ d(y,y),\] 
and we conclude that $\nabla d(y,y)=0$. Now, for $x>y$, by Theorem~\ref{thmn1}, we have that \[ \left\lvert \partial_2 d^+(x,y)\right\rvert\leq \left\lvert \partial_2^+ d(y,y)\right\rvert=0,\]  that is, $\partial_2 d(x,y)=0$. Hence, for $x>y$,
	$d(x,y)=-\int_{y}^x\partial_2 d(x,z)\operatorname{d} z=0$, which is not possible since $d$ is a distance.
\end{proof}

\begin{rem} Observe that, by the same reasoning as in Corollary~\ref{cor1}, $\partial_1^+ d(x,x)=\partial_2^+ d(x,x)>0$, $\partial_1^- d(x,x)=\partial_2^- d(x,x)>0$ for every $x\in {\mathbb R}$.
\end{rem}

\subsection{Conditions involving second order derivatives}

 Now, we will analyze some necessary conditions involving second order derivatives to guarantee that a function $d$ is a distance.

\begin{thm}
	Let $d\in{\mathcal C}({\mathbb R}^2,[0,\infty))$ be a distance such that 
	 $d|_X$ twice differentiable.	Then, $\partial_1^-(\partial_1^+ d)(x,x)\leq \partial_1^-(\partial_1^+ d)(x,y)$, for every $(x,y)\in {\mathbb R}^2$.
\end{thm}
\begin{proof}
Given that $d|_X$ is twice differentiable, for any $(x,y)\in{\mathbb R}$ and $h>0$ we have that, using the triangle inequality,


\begin{align*}-\partial_1^-(\partial_1^+d)(x,y)= & -\frac{1}{h} \left[
\partial_1^+ d(x-h,y)-\partial_1^+(x,y)+o(h)\right] \\
=& -\frac{1}{h} \left[\frac{1}{h} \left[
d(x,y)-d(x-h,y)+o(h)
\right]-\frac{1}{h}\left[
d(x+h,y)-d(x,y)+o(h)
\right]+o(h)
\right]\\
=&+ \frac{1}{h} \left[
\frac{1}{h} \left[
d(x+h,y)-2d(x,y)+d(x-h,y)
+o(h)
\right]
+o(h)
\right]\\
\leq &+ \frac{1}{h} \left[
\frac{1}{h} \left[
d(y,x)+d(x,x+h)-2d(x,y)+d(y,x)+d(x,x-h)
+o(h)
\right]
+o(h)
\right] \\
=& +\frac{1}{h}\left[
\frac{1}{h}\left[
d(x+h,x)+d(x-h,x)-2d(x,x)
+o(h)
\right]
+o(h)
\right]\\
=&-\frac{1}{h} \left[
\frac{1}{h} \left[ 
d(x,x)-d(x-h,x)
+o(h)
\right]-\frac{1}{h}\left[
d(x+h,x)-d(x,x)
+o(h)
\right]
+o(h)
\right]\\
=&-\frac{1}{h}\left[
\partial_1^+d(x-h,x)-\partial_1^+d(x,x)
+o(h)
\right] \\
=& -\partial_1^-(\partial_1^+ d)(x,x)
\end{align*}

After inverting the inequality, we get the result.
\end{proof}
\begin{rem}Observe that, if $x\ne y$, then $\partial_1^+(\partial_1^- d)(x,y)=-\partial_{11} d(x,y)$ and we have the inequality $\partial_{11} d(x,y)\leq -\partial_1^+(\partial_1^-d)(x,x)$. Furthermore, for the case $(x,y)\in X$, since $\partial_1^+(\partial_1^-d)(x,x))=\partial_2^+(\partial_1^-d)(x,x)$, $\partial_{11}d|_X(x,y)\leq \partial_{12}d|_X(x,x)$.
	\end{rem}

\section{Sufficient conditions}

In this section, we provide several sufficient results in order to ensure that a function $d:\mathbb{R}^2 \to \mathbb{R}$ defines a metric on $\mathbb{R}$. Of course, the hypotheses regarding the symmetry and sign of $d$ are obvious. Nevertheless, with respect to what hypotheses we can demand in order to obtain the triangle inequality, we make the following short discussion in this preamble.

We will assume $x<y<z$ without loss of generality, since the triangle inequality is evident in the case where at least two of these three numbers are equal (positive definiteness and symmetry are enough to conclude). It is important to have in mind these assumptions concerning the order of $x,y$ and $z$, since they will be used with high frequency along the whole document. In order to prove the triangle inequality, we need to show the following three inequalities:
\begin{equation}\label{ineqs}
\begin{aligned}
& \text{1. \quad} d(x,y)+d(y,z) \geq d(x,z), \\
& \text{2. \quad} d(x,z)+d(y,z) \geq d(x,y), \\
& \text{3. \quad} d(x,y)+d(x,z) \geq d(y,z).
\end{aligned}
\end{equation}

It is also important to realize that the nature of the first inequality is somehow distinct to the two other ones. The main reason is that the first inequality compares the ``distance'' from $x$ to $z$ with the sum of two ``distances'' pivoting via $y$. The fact that $y$ is the intermediate value in the order relation assumption $x<y<z$ plays a special role here. Nevertheless, the second and third inequalities are somehow symmetric, as we shall see along the proofs in this article.

In the first part of the section we will provide two versions of a useful lemma that, essentially, provides a sufficient condition for having the first inequality $d(x,y)+d(y,z) \geq d(x,z)$. In the second part, we will state and prove an initial version of these sufficiency theorems for $d$ being a distance. In each of these theorems, we add a different hypothesis that allows us to get $d(x,z)+d(y,z) \geq d(x,y)$ and $d(x,y)+d(x,z) \geq d(y,z)$. In the third part, we will weaken the assumption involving the smoothness of $d$. 

\subsection{The cross partial derivative and the triangle inequality}

We will provide two versions of a lemma that shows, roughly speaking, how a non-negative sign of the cross partial derivative outside of $\Delta$ implies $d(x,y)+d(y,z) \geq d(x,z)$ for any $x<y<z$. The main difference between these two versions is that $d$ is required to be continuous on $\mathbb{R}^2$ in Lemma~\ref{lem3}, but not in Lemma~\ref{lem4}. This extra assumption allows us to provide a simple proof of Lemma~\ref{lem3}, and a nice geometrical explanation of the issue via Figure~\ref{fig1}. Besides, this proof of Lemma~\ref{lem3} captures the essence of the idea that is needed to prove Lemma~\ref{lem4}. In this last case, the absence of continuity for $d$ forces us to make some technical considerations in our argumentation.

\begin{lem}\label{lem3}
Consider a function $d \in \mathcal{C}^2\left(\mathbb{R}^2 \backslash \Delta ,[0,\infty)\right) \cap \mathcal{C}\left(\mathbb{R}^2,[0,\infty)\right)$ fulfilling $\partial_{12} d (a,b) \geq 0$ for every $(a,b) \not \in \Delta$ and three given real numbers $x<y<z$. Then, we have the inequality $d(x,z) \leq d(x,y) + d(y,z)$.
\end{lem}

\begin{proof}\ 

	\emph{Idea:} The key for the proof of this lemma is to derive a differential inequality from the sign condition $\partial_{12}d \geq 0$, that implies the desired result after integration and a direct application of Fundamental Theorem of Calculus. In geometrical terms, see Figure~\ref{fig1}, we observe how, due to the symmetry of $d$, the integral of $\partial_2 d$ on the vertical black segment coincides with the integral of $\partial_1 d$ on the horizontal black segment. Then, since $\partial_{12} d \geq 0$, we know that $\partial_1 d$ increases with respect to increments in the second variable. Consequently, the integral of $\partial_1 d$ on the black segment will be bounded from above by the corresponding integral on the gray segment. 

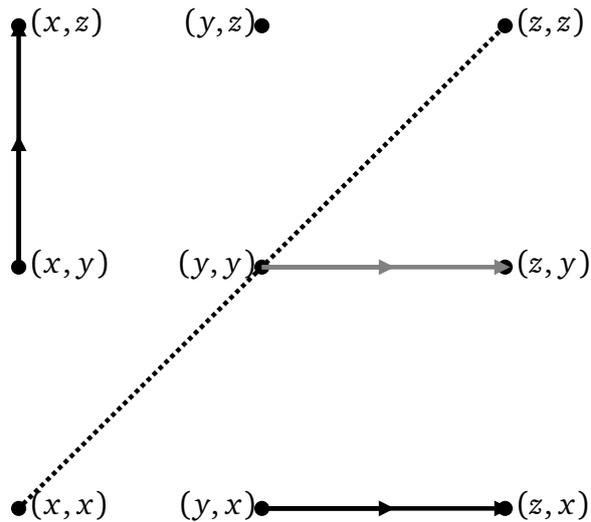
\begin{figure}[h]
	\centering
	\begin{tikzpicture}[x=.8mm, y=.8mm, inner xsep=0pt, inner ysep=0pt, outer xsep=0pt, outer ysep=0pt]
		\definecolor{F}{rgb}{0,0,0}
		\path[fill=F] (30.00,110.00) circle (1.00mm);
		\path[fill=F] (30.00,70.00) circle (1.00mm);
		\path[fill=F] (70.00,110.00) circle (1.00mm);
		\path[fill=F] (70.00,70.00) circle (1.00mm);
		\path[fill=F] (30.00,30.00) circle (1.00mm);
		\path[fill=F] (70.00,30.00) circle (1.00mm);
		\path[fill=F] (110.00,110.00) circle (1.00mm);
		\path[fill=F] (110.00,70.00) circle (1.00mm);
		\path[fill=F] (110.00,30.00) circle (1.00mm);
		\definecolor{T}{rgb}{0,0,0}
		\draw[T] (32.00,69.00) node[anchor=base west]{ {$(x,y)$}};
		\draw[T] (32.00,109.00) node[anchor=base west]{ {$(x,z)$}};
		\draw[T] (69.00,109.00) node[anchor=base east]{ {$(y,z)$}};
		\draw[T] (112.00,109.00) node[anchor=base west]{ {$(z,z)$}};
		\draw[T] (112.00,69.00) node[anchor=base west]{ {$(z,y)$}};
		\draw[T] (112.00,29.00) node[anchor=base west]{ {$(z,x)$}};
		\draw[T] (69.00,29.00) node[anchor=base east]{ {$(y,x)$}};
		\draw[T] (32.00,29.00) node[anchor=base west]{ {$(x,x)$}};
		\draw[T] (69.00,69.00) node[anchor=base east]{ {$(y,y)$}};
		\path[] (5.00,107.00) -- (-9.00,76.00);
		\definecolor{L}{rgb}{0,0,0}
		\path[line width=0.60mm, draw=L, dash pattern=on 0.60mm off 0.50mm] (110.00,110.00) -- (30.00,30.00);
		\definecolor{L}{rgb}{0.502,0.502,0.502}
		\path[line width=0.60mm, draw=L] (110.00,70.00) -- (70.00,70.00);
		\definecolor{F}{rgb}{0.502,0.502,0.502}
		\path[line width=0.60mm, draw=L, fill=F] (110.00,70.00) -- (108.60,70.70) -- (108.60,69.30) -- (110.00,70.00) -- cycle;
		\path[line width=0.60mm, draw=L] (91.00,70.00) -- (70.00,70.00);
		\path[line width=0.60mm, draw=L, fill=F] (91.00,70.00) -- (89.60,70.70) -- (89.60,69.30) -- (91.00,70.00) -- cycle;
		\definecolor{L}{rgb}{0,0,0}
		\path[line width=0.60mm, draw=L] (91.00,30.00) -- (70.00,30.00);
		\definecolor{F}{rgb}{0,0,0}
		\path[line width=0.60mm, draw=L, fill=F] (91.00,30.00) -- (89.60,30.70) -- (89.60,29.30) -- (91.00,30.00) -- cycle;
		\path[line width=0.60mm, draw=L] (110.00,30.00) -- (70.00,30.00);
		\path[line width=0.60mm, draw=L, fill=F] (110.00,30.00) -- (108.60,30.70) -- (108.60,29.30) -- (110.00,30.00) -- cycle;
		\path[line width=0.60mm, draw=L] (30.00,91.00) -- (30.00,70.00);
		\path[line width=0.60mm, draw=L, fill=F] (30.00,91.00) -- (29.30,89.60) -- (30.70,89.60) -- (30.00,91.00) -- cycle;
		\path[line width=0.60mm, draw=L] (30.00,110.00) -- (30.00,70.00);
		\path[line width=0.60mm, draw=L, fill=F] (30.00,110.00) -- (29.30,108.60) -- (30.70,108.60) -- (30.00,110.00) -- cycle;
	\end{tikzpicture}
	\caption{Representation of the different integration paths that lead to the inequality $d(x,y)+d(y,z) \geq d(x,z)$.}
	\label{fig1}
\end{figure}

From the calculus point of view, the proof for Lemma~\ref{lem3} is straightforward \[ d(x,z)-d(x,y)=d(z,x)-d(y,x)=\int_y^z \partial_1 d(s,x) \operatorname{d} s \leq \int_y^z \partial_1 d(s,y) \operatorname{d} s = d(z,y) = d(y,z),\]  since $\partial_1 d$ is increasing with respect to the second variable due to the condition $\partial_{12} d (a,b) \geq 0$.
\end{proof}

Now, we state a stronger version for the previous lemma, where we drop out the hypothesis regarding the continuity assumption for $d$. This generalization is relevant, since many renowned examples of distances on $\mathbb{R}$ are not induced by continuous functions, as we shall see in the last part of the paper. Observe that this loss of continuity on $\Delta$, and specifically at the point $(y,y)$, impedes us to use Fundamental Theorem of Calculus to claim that $\int_y^z \partial_1 d(s,y) \operatorname{d} s=d(z,y)-d(y,y)$. Hence, the technique consists in making a valid limit argument that does not need the continuity of $d$. Nevertheless, the main idea of this proof is, essentially, the same one as in Lemma~\ref{lem3}. In order to synthesize the argument, it will be convenient to establish a notation for certain functions, and to prove some properties regarding their monotonicity.

\begin{lem}\label{lem4}
Consider a function $d \in \mathcal{C}^2\left(\mathbb{R}^2 \backslash \Delta ,[0,\infty)\right)$ fulfilling the assumption $\partial_{12} d (a,b) \geq 0$ for every $(a,b) \not \in \Delta$ and three given real numbers $x<y<z$. Then, the function \[ G_{y,H}^z(\lambda)\coloneqq\int_y^z \partial_1 d(s,\lambda) \operatorname{d} s\]  is increasing on the intervals $(-\infty,y]$ and $[z,+\infty)$ and the function \[ G_{x,V}^y(\lambda)\coloneqq\int_x^y \partial_2 d(\lambda,s) \operatorname{d} s\]  is increasing on the intervals $(-\infty,x]$ and $[y,+\infty)$. As a consequence, we have the inequality $d(x,z) \leq d(x,y) + d(y,z)$.
\end{lem}

\begin{proof}
The key remark is that $\partial_1 d(s,\lambda) \leq \partial_1 d(s,\widetilde{\lambda})$ for any $s \in (y,z)$ and any $\lambda<\widetilde{\lambda}\leq y$, since $\partial_{12} d \geq 0$ on $R= (y,z) \times [\lambda,\widetilde{\lambda}] \subset (y,z) \times (-\infty,y]$ because $R \cap \Delta = \emptyset$. Therefore, $G_{y,H}^z(\lambda)$ is increasing on $(-\infty,y]$ and, analogously, it is also increasing on $[z,\infty)$. A similar argument applies in order to show the increasing character of $G_{x,V}^y(\lambda)$ on the intervals $(-\infty,x]$ and $[y,+\infty)$.

For the final claim, if we observe that $d$ is two times differentiable at any point of the closure $\overline{R}$ except at $(y,y)$, we can apply Fundamental Theorem of Calculus to deduce \[ d(z,\lambda)-d(y,\lambda) \leq d(z,\widetilde{\lambda})-d(y,\widetilde{\lambda}){\leq d(z,\widetilde{\lambda})},\]  for any $\lambda < \widetilde{\lambda} <y$. If we take $\lambda=x$, and let $\widetilde{\lambda} \to y$, the continuity of $d$ outside the diagonal together will imply \[ d(z,x)-d(y,x) \leq d(z,y),\] which is obviously equivalent to the desired inequality.
\end{proof}

\subsection{Initial statements for sufficient conditions}

We now state four sufficiency theorems, each of them implying that $d$ defines a distance under certain hypotheses, together with their corresponding proofs.

\begin{thm}\label{t_1} Consider a function $d \in \mathcal{C}^2\left(\mathbb{R}^2 \backslash \Delta ,[0,\infty)\right)$. Suppose that $d$ fulfills the following properties:
\begin{enumerate}
\item[H1.] $d(x,y) > 0$ for every $(x,y) \not \in \Delta$, and $d(x,x) =0$ for every $x \in \mathbb{R}$.
\item[H2.] $d(x,y) = d(y,x)$ for every $(x,y) \in \mathbb{R}^2$.
\item[H3.] $\partial_{12} d (x,y) \geq 0$ for every $(x,y) \not \in \Delta$.
\item[H4A.] For any fixed $a \in \mathbb{R}$, the function $d(\cdot, a)$ is non-increasing on the interval $(-\infty,a)$ and non-decreasing on the interval $(a, \infty)$.
\end{enumerate}
Then, $d$ defines a distance on $\mathbb{R}$.
\end{thm}

\begin{proof}
It is clear that we only need to check the triangle inequality for $x<y<z$. Besides, due to hypothesis H2, we observe that hypothesis H4A implies that $d(a,\cdot)$ is non-increasing on the interval $(-\infty,a)$ and non-decreasing on the interval $(a, \infty)$. On the one hand, $d(x,y)+d(y,z) \geq d(x,z)$ is a straightforward consequence of Lemma~\ref{lem4} due to hypothesis H3. On the other hand, derivation for the two last inequalities in~\eqref{ineqs} is immediate from H4A since distances are non negative, $d(x,z) \geq d(x,y)$, and $d(x,z) \geq d(y,z)$.
\end{proof}

\begin{thm}\label{t_2} Consider a function $d \in \mathcal{C}^2\left(\mathbb{R}^2 \backslash \Delta ,[0,\infty)\right)$ fulfilling the following properties:
\begin{enumerate}
\item[H1.] $d(x,y) > 0$ for every $(x,y) \not \in \Delta$, and $d(x,x) =0$ for every $x \in \mathbb{R}$.
\item[H2.] $d(x,y) = d(y,x)$ for every $(x,y) \in \mathbb{R}^2$.
\item[H3.]$\partial_{12} d (x,y) \geq 0$ for every $(x,y) \not \in \Delta$.
\item[H4B.] $\lim_{\lambda \to +\infty} [d(b,\lambda)-d(a,\lambda) ]\leq \lim_{\lambda \to -\infty} [d(b,\lambda)-d(a,\lambda)]$ for every pair $(a,b)$ with $a<b$, where both limits are finite.
\end{enumerate}
Then, $d$ defines a distance on $\mathbb{R}$.
\end{thm}
\emph{Idea.} 
In geometrical terms --see Figure~\ref{fig2}, left-- if we consider a horizontal segment from $(y,\lambda)$ to $(z,\lambda)$, the integral of $\partial_1 d$ along the oriented segment increases when the height $\lambda$ increases, with the only caution that the segment cannot cut $\Delta$. So, instead of cutting $\Delta$, the idea consists in ``passing through infinity'', as it will be explained in the next paragraph. The geometrical explanation for Figure~\ref{fig2}, right, is the same one, but changing vertical and horizontal roles.
In terms of Figure~\ref{fig2}, we want to prove that the integral of the horizontal/vertical partial derivative of $d$ along the gray oriented segment is greater than the analogous one on the black oriented segment.
\begin{figure}[h]
	\centering
	\begin{tikzpicture}[x=0.80mm, y=0.80mm, inner xsep=0pt, inner ysep=0pt, outer xsep=0pt, outer ysep=0pt]
		\definecolor{F}{rgb}{0,0,0}
		\path[fill=F] (30.00,110.00) circle (1.00mm);
		\path[fill=F] (30.00,70.00) circle (1.00mm);
		\path[fill=F] (70.00,110.00) circle (1.00mm);
		\path[fill=F] (70.00,70.00) circle (1.00mm);
		\path[fill=F] (30.00,30.00) circle (1.00mm);
		\path[fill=F] (70.00,30.00) circle (1.00mm);
		\path[fill=F] (110.00,110.00) circle (1.00mm);
		\path[fill=F] (110.00,70.00) circle (1.00mm);
		\path[fill=F] (110.00,30.00) circle (1.00mm);
		\definecolor{T}{rgb}{0,0,0}
		\draw[T] (32.00,69.00) node[anchor=base west]{ {$(x,y)$}};
		\draw[T] (32.00,109.00) node[anchor=base west]{ {$(x,z)$}};
		\draw[T] (69.00,109.00) node[anchor=base east]{ {$(y,z)$}};
		\draw[T] (112.00,109.00) node[anchor=base west]{ {$(z,z)$}};
		\draw[T] (112.00,69.00) node[anchor=base west]{ {$(z,y)$}};
		\draw[T] (112.00,29.00) node[anchor=base west]{ {$(z,x)$}};
		\draw[T] (69.00,29.00) node[anchor=base east]{ {$(y,x)$}};
		\draw[T] (32.00,29.00) node[anchor=base west]{ {$(x,x)$}};
		\draw[T] (69.00,69.00) node[anchor=base east]{ {$(y,y)$}};
		\definecolor{L}{rgb}{0,0,0}
		\path[line width=0.60mm, draw=L, dash pattern=on 0.60mm off 0.50mm] (110.00,110.00) -- (30.00,30.00);
		\definecolor{L}{rgb}{0,0,0}
		\path[line width=0.60mm, draw=L] (110.00,110.00) -- (70.00,110.00);
		\definecolor{F}{rgb}{0,0,0}
		\path[line width=0.60mm, draw=L, fill=F] (110.00,110.00) -- (108.60,110.70) -- (108.60,109.30) -- (110.00,110.00) -- cycle;
		\path[line width=0.60mm, draw=L] (91.00,110.00) -- (70.00,110.00);
		\path[line width=0.60mm, draw=L, fill=F] (91.00,110.00) -- (89.60,110.70) -- (89.60,109.30) -- (91.00,110.00) -- cycle;
		\definecolor{L}{rgb}{0.502,0.502,0.502}
		\path[line width=0.60mm, draw=L] (91.00,30.00) -- (70.00,30.00);
		\definecolor{F}{rgb}{0.502,0.502,0.502}
		\path[line width=0.60mm, draw=L, fill=F] (91.00,30.00) -- (89.60,30.70) -- (89.60,29.30) -- (91.00,30.00) -- cycle;
		\path[line width=0.60mm, draw=L] (110.00,30.00) -- (70.00,30.00);
		\path[line width=0.60mm, draw=L, fill=F] (110.00,30.00) -- (108.60,30.70) -- (108.60,29.30) -- (110.00,30.00) -- cycle;
	\end{tikzpicture}\qquad
	\begin{tikzpicture}[x=0.80mm, y=0.80mm, inner xsep=0pt, inner ysep=0pt, outer xsep=0pt, outer ysep=0pt]
		\definecolor{F}{rgb}{0,0,0}
		\path[fill=F] (30.00,110.00) circle (1.00mm);
		\path[fill=F] (30.00,70.00) circle (1.00mm);
		\path[fill=F] (70.00,110.00) circle (1.00mm);
		\path[fill=F] (70.00,70.00) circle (1.00mm);
		\path[fill=F] (30.00,30.00) circle (1.00mm);
		\path[fill=F] (70.00,30.00) circle (1.00mm);
		\path[fill=F] (110.00,110.00) circle (1.00mm);
		\path[fill=F] (110.00,70.00) circle (1.00mm);
		\path[fill=F] (110.00,30.00) circle (1.00mm);
		\definecolor{T}{rgb}{0,0,0}
		\draw[T] (32.00,69.00) node[anchor=base west]{ {$(x,y)$}};
		\draw[T] (32.00,109.00) node[anchor=base west]{ {$(x,z)$}};
		\draw[T] (69.00,109.00) node[anchor=base east]{ {$(y,z)$}};
		\draw[T] (112.00,109.00) node[anchor=base west]{ {$(z,z)$}};
		\draw[T] (112.00,69.00) node[anchor=base west]{ {$(z,y)$}};
		\draw[T] (112.00,29.00) node[anchor=base west]{ {$(z,x)$}};
		\draw[T] (69.00,29.00) node[anchor=base east]{ {$(y,x)$}};
		\draw[T] (32.00,29.00) node[anchor=base west]{ {$(x,x)$}};
		\draw[T] (69.00,69.00) node[anchor=base east]{ {$(y,y)$}};
		\definecolor{L}{rgb}{0,0,0}
		\path[line width=0.60mm, draw=L, dash pattern=on 0.60mm off 0.50mm] (110.00,110.00) -- (30.00,30.00);
		\path[line width=0.60mm, draw=L] (110.00,49.00) -- (110.00,70.00);
		\path[line width=0.60mm, draw=L, fill=F] (110.00,51.00) -- (109.30,49.60) -- (110.70,49.60) -- (110.00,51.00) -- cycle;
		\path[line width=0.60mm, draw=L] (110.00,70.00) -- (110.00,30.00);
		\path[line width=0.60mm, draw=L, fill=F] (110.00,70.00) -- (109.30,68.60) -- (110.70,68.60) -- (110.00,70.00) -- cycle;
		\definecolor{L}{rgb}{0.502,0.502,0.502}
		\path[line width=0.60mm, draw=L] (30.00,51.00) -- (30.00,30.00);
		\definecolor{F}{rgb}{0.502,0.502,0.502}
		\path[line width=0.60mm, draw=L, fill=F] (30.00,51.00) -- (29.30,49.60) -- (30.70,49.60) -- (30.00,51.00) -- cycle;
		\path[line width=0.60mm, draw=L] (30.00,70.00) -- (30.00,30.00);
		\path[line width=0.60mm, draw=L, fill=F] (30.00,70.00) -- (29.30,68.60) -- (30.70,68.60) -- (30.00,70.00) -- cycle;
	\end{tikzpicture}
	\caption{Comparison of the horizontal (left) and vertical (right) segments.}
	\label{fig2}
\end{figure}
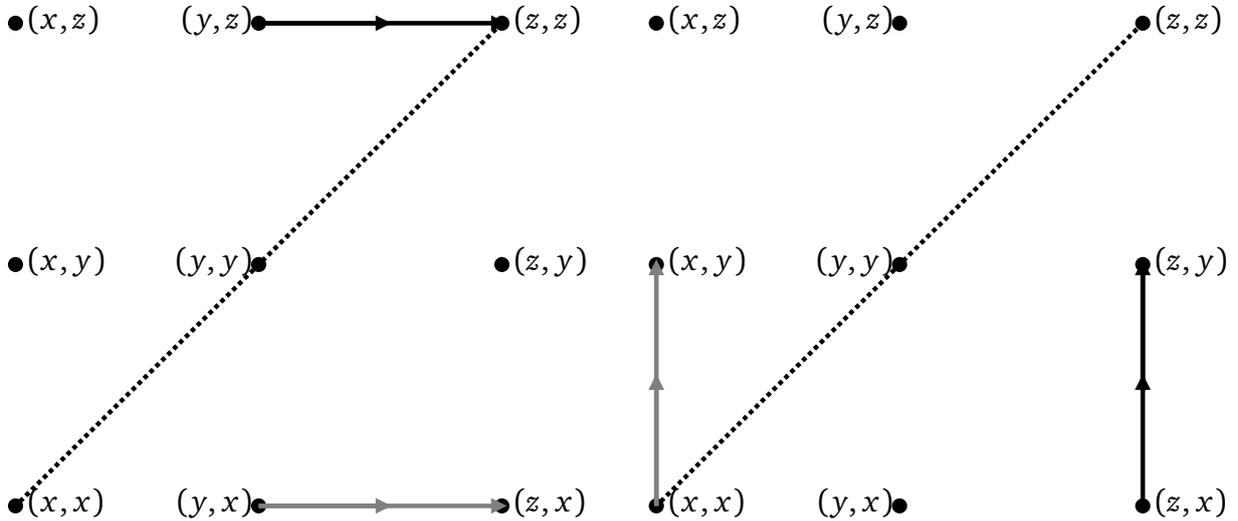

\begin{proof}
As in the previous result, from the first three hypotheses we derive the non-negativity, symmetry, and 1 in~\eqref{ineqs}. Thus, if $x<y<z$, it suffices to see that
\begin{itemize}
\item $d(x,z)+d(y,z) \geq d(x,y)$,
\item $d(x,y)+d(x,z) \geq d(y,z)$.
\end{itemize}
First, if we recall the definition $G_{y,H}^z(\lambda)\coloneqq\int_y^z \partial_1 d(s,\lambda) \operatorname{d} s$ made in Lemma~\ref{lem4}, we know that $G_{y,H}^z$ increases on $(-\infty,y]$ and $[z,\infty)$. Besides, since $d$ is smooth enough outside of $\Delta$, we can apply the Fundamental Theorem of Calculus to the integral defined by $G_{y,H}^z(\lambda)$ for $\lambda \in (-\infty,y) \cup (z,\infty)$. In particular, we are interested in the cases where $\lambda \to \infty$ or $\lambda \to -\infty$, since H4B can be stated as
\begin{equation} \label{eq1}
\lim_{\lambda \to +\infty} G_{y,H}^z(\lambda) \leq \lim_{\lambda \to -\infty} G_{y,H}^z(\lambda),
\end{equation}
implying $G_{y,H}^z(x) \geq G_{y,H}^z(z)$. Analogously, if we recall the definition $G_{x,V}^y(\lambda):=\int_x^y \partial_2 d(\lambda,s) \operatorname{d} s$, due to the symmetry of $d$, we immediately derive
\begin{equation} \label{eq2}
\lim_{\lambda \to +\infty} G_{x,V}^y(\lambda) \leq \lim_{\lambda \to -\infty} G_{x,V}^y(\lambda),
\end{equation}
implying $G_{x,V}^y(x) \geq G_{x,V}^y(z)$. We simply observe that
\begin{align*}
G_{y,H}^z(x) \geq G_{y,H}^z(z) \Leftrightarrow d(z,x)-d(y,x) \geq -d(y,z) \Leftrightarrow & d(x,z)+d(y,z) \geq d(x,y),\\
G_{x,V}^y(x) \geq G_{x,V}^y(z) \Leftrightarrow d(x,y) \geq d(z,y)-d(z,x) \Leftrightarrow & d(x,y)+d(x,z) \geq d(y,z),
\end{align*}
and we are finished.
\end{proof}

\begin{thm}\label{t_3} Consider a function $d \in \mathcal{C}^2\left(\mathbb{R}^2 \backslash \Delta ,[0,\infty)\right)$ fulfilling the following properties:
\begin{enumerate}
\item[H1.] $d(x,y) > 0$ for every $(x,y) \not \in \Delta$, and $d(x,x) =0$ for every $x \in \mathbb{R}$.
\item[H2.] $d(x,y) = d(y,x)$ for every $(x,y) \in \mathbb{R}^2$.
\item[H3.] $\partial_{12} d(x,y) \geq 0$ for every $(x,y) \not \in \Delta$.
\item[H4C.] $\nabla d(x,y) \to 0$ when $(x,y) \to \infty$.
\end{enumerate}
Then, $d$ defines a distance on $\mathbb{R}$.
\end{thm}

\begin{proof} In order to prove this theorem, it suffices to see how H4B is derived from H4C. Since the gradient tends to zero, for any given $\varepsilon >0$, it is possible to make $\Vert \nabla d \Vert < \varepsilon$ outside of a big enough square $[-l,l] \times [-l,l]$. Hence, after considering any $\lambda$ such that $\vert \lambda \vert > l$, we have that $\vert G_{a,H}^b (\lambda) \vert \leq \varepsilon \cdot (b-a)$. Thus, H4B would be automatically fulfilled, since it would read $0 \leq 0$.
\end{proof}

\begin{thm}\label{t_4} Consider a function $d \in \mathcal{C}^2\left(\mathbb{R}^2 \backslash \Delta ,[0,\infty)\right)$ fulfilling the following properties:
\begin{enumerate}
\item[H1.] $d(x,y) > 0$ for every $(x,y) \not \in \Delta$, and $d(x,x) =0$ for every $x \in \mathbb{R}$.
\item[H2.] $d(x,y) = d(y,x)$ for every $(x,y) \in \mathbb{R}^2$.
\item[H3.] $\partial_{12} d (x,y) \geq 0$ for every $(x,y) \not \in \Delta$.
\item[H4D.] We have that $\lim_{\lambda \to -\infty} d(c,\lambda) = \lim_{\lambda \to \infty} d(c,\lambda)\in{\mathbb R}$ for any $c \in \mathbb{R}$.
\end{enumerate}
Then, $d$ defines a distance on $\mathbb{R}$.
\end{thm}
\begin{proof} It is an immediate consequence of Theorem~\ref{t_2}, since H4D implies H4B in an obvious way.
\end{proof}
\begin{rem} Due to the symmetry property, hypothesis $4D$ is obviously equivalent to what we could call hypothesis H4D', that would read as
\begin{itemize}
\item[\textit{H4D'.}] We have that $\lim_{\lambda \to -\infty} d(\lambda,c) = \lim_{\lambda \to \infty} d(\lambda,c)\in{\mathbb R}$ for any $c \in \mathbb{R}$.
\end{itemize}

We have made explicit the previous remark since H4D and H4D' imply that $d$ can be extended to a class two map on the periodic domain of the form $M=\mathbb{S}^1 \times \mathbb{S}^1 \setminus \{(\alpha,\alpha) \in \mathbb{S}^1 \times \mathbb{S}^1:\alpha \in \mathbb{S}^1\}$, provided that the matching at the infinity points induced by H4D and H4D' is smooth enough. Since $M$ is diffeomorphic to a cylinder, a possible way to produce distances on $\mathbb{R}$ would be, roughly speaking, to find class two scalar fields on a cylinder that are positive (hypothesis 1), symmetric (hypothesis 2) and with non-negative cross partial derivative after applying the already mentioned diffeomorphism (hypothesis 3).
\end{rem}

\subsection{An extension for sufficient conditions}
Before applying the previous theorems to some examples, it will be convenient to weaken their hypotheses, specially the one involving the required smoothness for $d$. In order to do so, first, we take into account the following trivial remark.

\begin{rem}\label{obsini} Suppose that $h : \mathbb{R} \to \mathbb{R}$ is a bijective map. Then, $d(x,y)$ defines a distance on $\mathbb{R}$ if and only if $(d \circ (h \times h))(x,y)=d(h(x),h(y))$ defines a distance on $\mathbb{R}$. In particular, $d$ fulfills H1 and H2 in Theorems~\ref{t_1},~\ref{t_2},~\ref{t_3}, and~\ref{t_4} if and only if $(d \circ (h \times h))$ fulfills hypotheses H1 and H2.
\end{rem}

The previous remark states, essentially, that being a distance does not depend on the coordinates that we are considering. We highlight that, in principle, this map $h$ would not need to be even continuous, measurable or to have any nice property. Nevertheless, the interest of the previous remark is that, in some examples and for some points $(x,y)$ outside the diagonal, the distance $d$ may not be regular enough in order to apply any result of the previous section. This problem can be avoided after considering the distance $d \circ (h \times h)$ for a suitable smooth choice of $h$, instead of simply considering the distance $d$. In practice, the choice for the function $h$ will be $h(x)=x^{2n+1}$ for a suitable natural number $n\in \mathbb{N}$. The reason for this choice is that, in many expressions, we have addends like $\vert x \vert^p$ that are not smooth enough in order to apply the previous theorems when $p>0$ is too low, but we can make $\vert h(x) \vert^p$ to be smooth enough after choosing a sufficiently large value for $n$. In this sense, we take into account the following well known remark.

\begin{rem} For any fixed $p>0$, the regularity of the function $g(x)=\vert h(x) \vert^{p}=\vert x \vert^{(2n+1)p}$ increases with respect to $n$. Specifically, $g \in \mathcal{C}^m \left(\mathbb{R}\right)$, where $m={\lceil (2n+1)p \rceil-1}$. In particular, $g \in \mathcal{C}^2 \left(\mathbb{R}\right)$ whenever $(2n+1)p>2$.
\end{rem}

The consideration made in the previous remark, and the fact that such a choice for $h$ is bijective and increasing, are the motivation for the next two lemmas. First, we state and demonstrate a sufficient condition ensuring that $d \circ (h \times h)$ fulfills H3 on $\mathbb{R}^2 \setminus \Delta$, provided that $d$ satisfies H3 on a bit smaller set.

\begin{lem} \label{lem1} Consider a function $d:{\mathbb R}^2 \to [0,\infty))$ such that $d \in \mathcal{C}^2\left(\mathbb{R}^2 \backslash (\Delta \cup \Lambda),[0,\infty)\right)$, where the set $\mathbb{R}^2\setminus (\Delta \cup \Lambda)$ is dense in $\mathbb{R}^2\setminus \Delta$. Assume that it exists an increasing bijective differentiable map $h:\mathbb{R} \to \mathbb{R}$ such that $d \circ (h \times h) \in \mathcal{C}^2\left(\mathbb{R}^2 \backslash \Delta,[0,\infty)\right)$. If $\partial_{12} d$ is non-negative on $\mathbb{R}^2\setminus (\Delta \cup \Lambda)$, then $\partial_{12} (d \circ (h \times h))$ is non-negative outside of $\Delta$.
\end{lem}

\begin{proof}
First observe that, since $h$ is bijective increasing and continuous, it is a homeomorphism. Therefore, the function $\varphi(x,y)=(h(x),h(y))$, where $(x,y)\in{\mathbb R}^2$, is a homeomorphism as well and fulfills $\varphi(\Delta)=\Delta$. Hence, since $\mathbb{R}^2\setminus (\Delta \cup \Lambda)$ is dense in $\mathbb{R}^2\setminus \Delta$, $\varphi^{-1}(\mathbb{R}^2\setminus (\Delta \cup \Lambda))$ is dense in $\varphi^{-1}(\mathbb{R}^2\setminus \Delta)=\mathbb{R}^2\setminus \Delta$.

Due to the chain rule, we have that \[ \partial_{12} (d \circ \varphi)(x,y)=\partial_{12} (d \circ (h \times h))(x,y)= h'(x) \cdot h'(y) \cdot (\partial_{12} d)(h(x),h(y)),\] and this expression is valid whenever ${\varphi(x,y)} \in \mathbb{R}^2\setminus (\Delta \cup \Lambda)$, {that is, $(x,y)\in\varphi^{-1}(\mathbb{R}^2\setminus (\Delta \cup \Lambda))$}. Besides, since $h$ is increasing and $\partial_{12}d$ is non-negative, we conclude that $\partial_{12} (d \circ (h \times h))$ is non-negative whenever {$(x,y)\in\varphi^{-1}(\mathbb{R}^2\setminus (\Delta \cup \Lambda))$}. Finally, since $d \circ (h \times h) \in \mathcal{C}^2\left(\mathbb{R}^2 \backslash \Delta,[0,\infty)\right)$, and $\varphi^{-1}(\mathbb{R}^2\setminus (\Delta \cup \Lambda))$ is dense in $\varphi^{-1}(\mathbb{R}^2\setminus \Delta)=\mathbb{R}^2\setminus \Delta$, we conclude that $\partial_{12} (d \circ (h \times h))$ is non-negative outside of $\Delta$.
\end{proof}

Finally, we state a second lemma ensuring that if $d$ fulfills one of the versions of H4, so does $d \circ (h \times h)$. We deliberately exclude hypothesis H4C, since the change of coordinates induced by $h$ can break the vanishing property for the gradient of $d$ at infinity.

\begin{lem} \label{lem2} Consider a function $d\in{\mathcal C}({\mathbb R}^2,[0,\infty))$, together with an increasing bijection $h: \mathbb{R} \to \mathbb{R}$.
\begin{itemize}
\item If $d$ fulfills H4A, then $d \circ (h \times h)$ fulfills H4A.
\item If $d$ fulfills H4B, then $d \circ (h \times h)$ fulfills H4B.
\item If $d$ fulfills H4D, then $d \circ (h \times h)$ fulfills H4D.
\end{itemize}
\end{lem}

\begin{proof}
For the first part, we have that, for any fixed $a \in \mathbb{R}$, the function $d(\cdot, a)$ is non-increasing on the interval $(-\infty,a)$ and non-decreasing on the interval $(a, \infty)$. Since $h$ is bijective and it preserves the order in the real line, $d \circ (h \times h)$ fulfills H4A.

For the second part, we have that $\lim_{\lambda \to +\infty} [d(b,\lambda)-d(a,\lambda) ]\leq \lim_{\lambda \to -\infty} [d(b,\lambda)-d(a,\lambda)]$ for every pair $(a,b)$ with $a<b$. Since $h$ is bijective and it preserves the order in the real line, $d \circ (h \times h)$ fulfills H4B.

For the last part, we have that $\lim_{\lambda \to -\infty} d(c,\lambda) = \lim_{\lambda \to \infty} d(c,\lambda)$ for any $c \in \mathbb{R}$, and that this value is finite. Since $h$ is bijective and it preserves the order in the real line, $d \circ (h \times h)$ fulfills H4D.
\end{proof}

As a consequence of all the previously exposed material, the main result of this part of the section can be stated and proved as follows.

\begin{thm}\label{t_5} Consider a function $d: \mathbb{R}^2 \to [0,\infty))$ such that $d \in \mathcal{C}^2\left(\mathbb{R}^2 \backslash (\Delta \cup \Lambda),[0,\infty)\right)$, where $\mathbb{R}^2\setminus (\Delta \cup \Lambda)$ is dense in $\mathbb{R}^2\setminus \Delta$. Assume that it exists an increasing bijective differentiable map $h:\mathbb{R} \to \mathbb{R}$ such that $d \circ (h \times h) \in \mathcal{C}^2\left(\mathbb{R}^2 \backslash \Delta,[0,\infty)\right)$. Finally, suppose also that $d$ fulfills the following properties:
\begin{enumerate}
\item[H1.] $d(x,y) > 0$ for every $(x,y) \not \in \Delta$, and $d(x,x) =0$ for every $x \in \mathbb{R}$.
\item[H2.] $d(x,y) = d(y,x)$ for every $(x,y) \in \mathbb{R}^2$.
\item[H3'.] $\partial_{12} d (x,y) \geq 0$ for every $(x,y) \not \in \Delta \cup \Lambda$.
\item[H4.] The function $d$ fulfills at least one of the hypotheses $4A$, $4B$ or $4D$ in Theorems~\ref{t_1},~\ref{t_2} or~\ref{t_4}.
\end{enumerate}
Then, $d$ defines a distance on $\mathbb{R}$.
\end{thm}
\begin{proof} Depending on whether $d$ fulfills hypothesis H4A, H4B or H4D we shall use Theorem~\ref{t_1},~\ref{t_2} or~\ref{t_4} to conclude. First, if $d$ fulfills H1 and H2 so does $d \circ (h \times h)$, due to Remark~\ref{obsini}. Second, Lemma~\ref{lem1} allows to deduce H3' for $d \circ (h \times h)$. Third, at least one of the hypotheses H4A, H4B or H4D is fulfilled by $d$ and, due to Lemma~\ref{lem2}, also by $d \circ (h \times h)$. We highlight that $d \circ (h \times h)$ is a continuous map and that its cross partial derivative is smooth enough outside the diagonal because of the hypotheses. Therefore, $d \circ (h \times h)$ is a metric and, consequently, $d$ is a metric due to Remark~\ref{obsini}.
\end{proof}

\section{Examples}
Here we present some classical examples of distances to which the previous criteria can be applied to provide a proof that they are distances. Initially, we provide five examples of five candidates to distances $d$ and, in each case, via Theorems~\ref{t_1},~\ref{t_2},~\ref{t_3},~\ref{t_4} or the general version in~\ref{t_5}, we deduce that $d$ is indeed a distance. More examples of known distances can be found in \cite{Deza2013}.
%
\begin{exa}[Concave translation invariant metric] The function $d(x,y)=g(\left\lvert y-x\right\rvert)$ is a distance, whenever $g\in\mathcal{C}^2([0,\infty))$ is concave, $g(0)=0$ and $g(x)>0$ if $x>0$.\par

Let us check the hypothesis of Theorem~\ref{t_1}. We have the required regularity outside of the diagonal, and also positive definitness and symmetry. If $(x,y)\notin\Delta$,
\[ \partial_{12}d(x,y)= \begin{dcases}-g''(y-x), & \text{ if } x<y , \\ -g''(x-y), & \text{ if } x> y,\end{dcases}\] 
and since $g$ is concave and twice differentiable on $(0,\infty)$, $\partial_{12}d(x,y)\geq 0$ for $(x,y)\notin\Delta$. Lastly, such $g$ is necessarily a non-decreasing function, so H4A in Theorem~\ref{t_1} is also satisfied, and $d$ is a metric.

Recall that we had already studied translation invariant metrics in Section 2. Therefore, even without the differentiability assumption, we could have stated that a positive concave function $g$ defined on $(0,\infty)$ is necessarily subadditive and non-decreasing. Hence, by Corollary~\ref{c1}, the function defined as $d(x,y)=g(\left\lvert y-x\right\rvert)$ is a distance.
\end{exa}

\begin{exa}[$p$-relative metric] \label{e_2} Given $p\in[1,\infty)$,
the function \[ \displaystyle d(x,y)=	\begin{dcases}\frac{\left\lvert y-x\right\rvert}{\left(\left\lvert x\right\rvert^p+\left\lvert y\right\rvert^p\right)^{\frac{1}{p}}}, & (x,y)\in{\mathbb R}^2,\ (x,y)\ne 0,\\ 0, & (x,y)=0,\end{dcases}\]  is a distance.
\begin{figure}[h!]
	\centering
	\includegraphics[width=.5\linewidth]{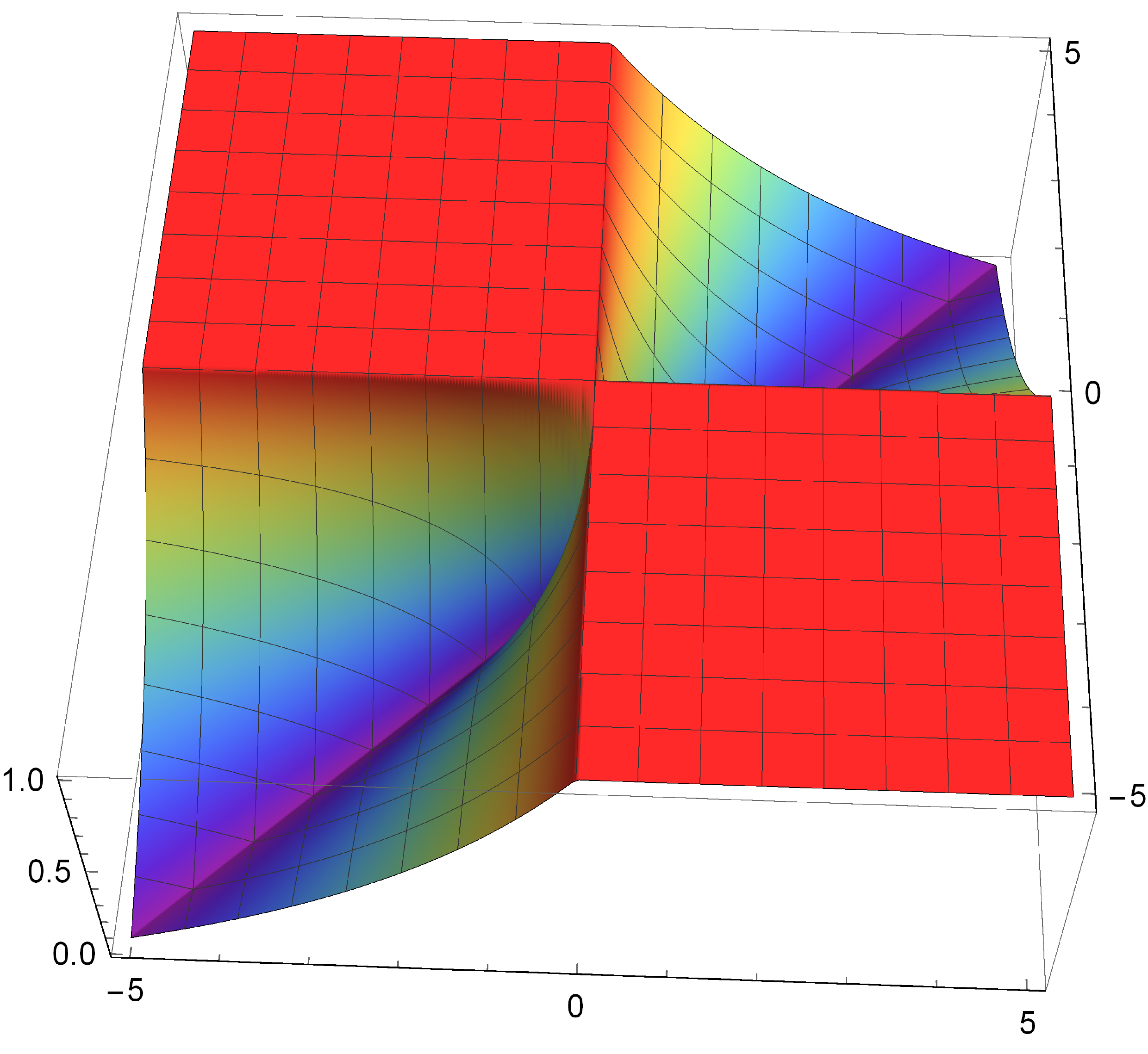}\includegraphics[width=.5\linewidth]{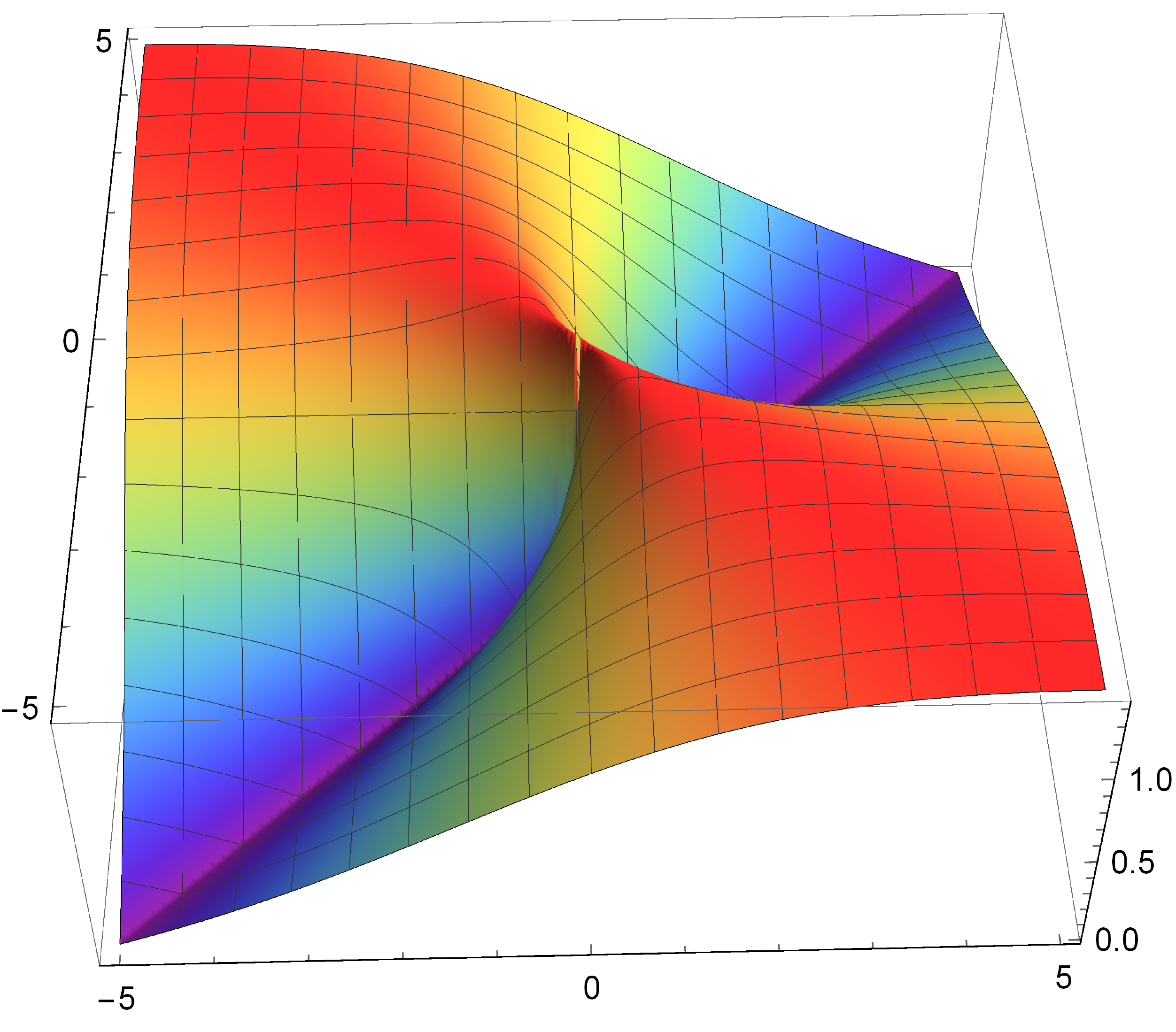}\caption{$1$-relative (left) and $2$-relative (right) metrics.}\label{prl}
\end{figure}

In order to apply any of the previous results, the only possible option is to check the hypotheses of Theorem~\ref{t_5}, since $d$ is not smooth on the set $\Lambda$ of the points $(x,y)$ where $x \cdot y=0$, at least when $p$ is small enough.

First, if we select $h(x)=x^{2n+1}$ in a such a way that $(2n+1)p>2$ we are able to ensure the condition $d \circ (h \times h) \in \mathcal{C}^2\left(\mathbb{R}^2 \backslash \Delta,[0,\infty)\right)$. In particular, we highlight that, in contrast to what happens with $d$, $d \circ (h \times h)$ is two times continuously differentiable on $\mathbb{R}^2\setminus \Delta$.

As usual, positive definiteness and symmetry are clear. Besides, hypothesis H4D is immediate to check, since \[ \lim_{\lambda \to - \infty} d(c,\lambda) = \lim_{\lambda \to - \infty} \frac{\left\lvert \lambda-c\right\rvert}{\left(\left\lvert c\right\rvert^p+\left\lvert \lambda\right\rvert^p\right)^{\frac{1}{p}}}= 1 =\lim_{\lambda \to + \infty} \frac{\left\lvert \lambda-c\right\rvert}{\left(\left\lvert c\right\rvert^p+\left\lvert \lambda\right\rvert^p\right)^{\frac{1}{p}}}= \lim_{\lambda \to +\infty} {d}(c,\lambda).\] Thus, we only need to check H3' or, in other words, that the cross partial derivative of $d$ is non-negative outside $\Delta \cup \Lambda$. Due to symmetry we can assume $x>y$. An straightforward computation for the three cases ($x>y> 0$, $x > 0>y$, $0 > x > y$) gives \[ \partial_{12}{d}(x,y)=\frac{\operatorname{sgn}(x)\left\lvert x\right\rvert^{2p-1}-\operatorname{sgn}(y)\left\lvert y\right\rvert^{2p-1}}{\left(\vert x \vert^p + \vert y \vert^p \right)^{2+\frac{1}{p}}} +p \operatorname{sgn}(xy) \frac{ \vert x \vert - \vert y \vert}{\left(\vert x \vert^p + \vert y \vert^p \right)^{2+\frac{1}{p}}} \vert x y \vert^{p-1}.\]  It is clear that $\partial_{12}{d}(x,y) \geq 0$ on the open ray of argument $-\frac{\pi}{4}$ Besides, this is also true for half of the first quadrant (where $x>y>0$), and for half of the fourth quadrant (where $0>-x>y$), since both addends involved in $\partial_{12} f$ are clearly positive.

The inequality for the two pending cases is derived as follows. For the case of half of the third quadrant $0>x>y$, we can write $\vert x \vert = \lambda \vert y \vert$ for some $0<\lambda<1$. After this substitution, cancellation of denominators, and taking into account the value for the sign function, we need to check that \[ \vert y \vert^{2p-1}(1-\lambda^{2p-1}) \geq p \vert y \vert^{2p-1} (1-\lambda) \lambda^{p-1}.\]  Equivalently, we have to explain why \[  1+p\lambda^p \geq p \lambda^{p-1}+\lambda^{2p-1}.\]  If we define the function \[ h(x)=(x+1)\lambda^x + (p-x)\lambda^{p+x},\]  since $h(0)= 1+p\lambda^p $ and $h(p-1)=p \lambda^{p-1}+\lambda^{2p-1}$, it is enough to prove that $h(0) \geq h(p-1)$. This is straightforward since $\log(\lambda) < 0$ and, for any $x \in (0,p-1)$, we have that \[ h'(x) = \log(\lambda)((x+2)\lambda^x+(p-x-1)\lambda^{p+x}) < 0.\] 

The pending case for half of the fourth quadrant $0>y>-x$ is similar to the previous one. We can write $\vert y \vert = \lambda \vert x \vert$ for some $0<\lambda<1$. We need to check that \[ \vert x \vert^{2p-1}(1+\lambda^{2p-1}) \geq p \vert x \vert^{2p-1} (1-\lambda) \lambda^{p-1}.\]  This is equivalent to showing that \[  1+p\lambda^p \geq p \lambda^{p-1}-\lambda^{2p-1},\]  but we have already verified this for the case where $\lambda^{2p-1}$ carries a plus sign.

	\end{exa}
\begin{exa}[Relative metric]
	The function 
	 \[ \displaystyle d(x,y)=	\begin{dcases}\frac{\left\lvert y-x\right\rvert}{\max \{\left\lvert x\right\rvert,\left\lvert y\right\rvert\}}, & (x,y)\in{\mathbb R}^2,\ (x,y)\ne 0,\\ 0, & (x,y)=0,\end{dcases}\]  is a distance.
\begin{figure}[h!]
	\centering
		\includegraphics[width=.5\linewidth]{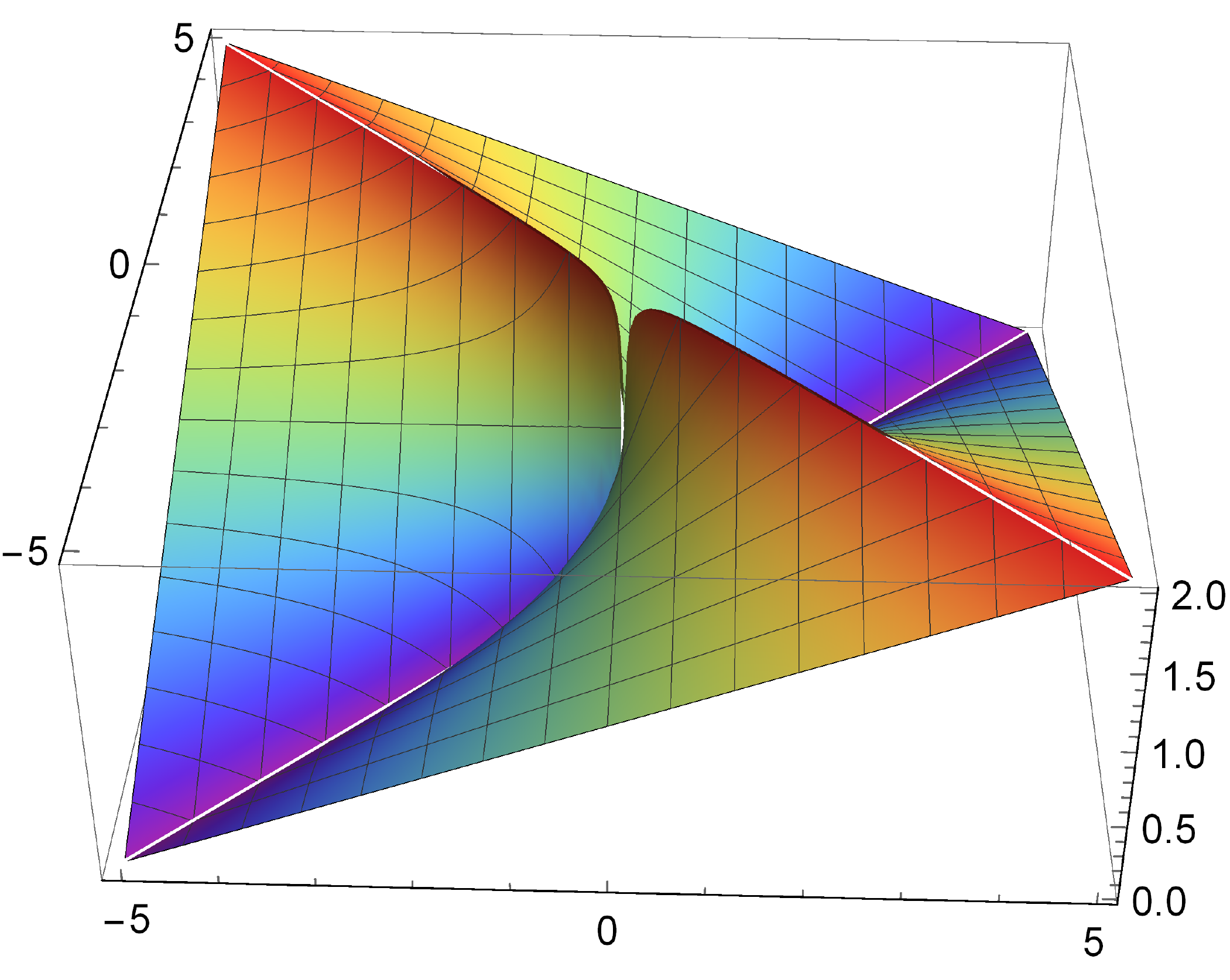}\caption{Relative metric.}\label{rl}
\end{figure}

The relative metric is the pointwise limit of the $p$-relative metrics when $p \to \infty$. Since positive definiteness, symmetry and triangle inequality are clearly preserved by taking limits on $p$, the function $d$ is a distance.

	\end{exa}
\begin{exa}[Chordal metric] The function $\displaystyle d(x,y)=\frac{2\left\lvert y-x\right\rvert}{\sqrt{1+x^2} \sqrt{1+y^2}}$, $x,y\in{\mathbb R}$ is a distance.
	\begin{figure}[h!]
		\centering
		\includegraphics[width=.5\linewidth]{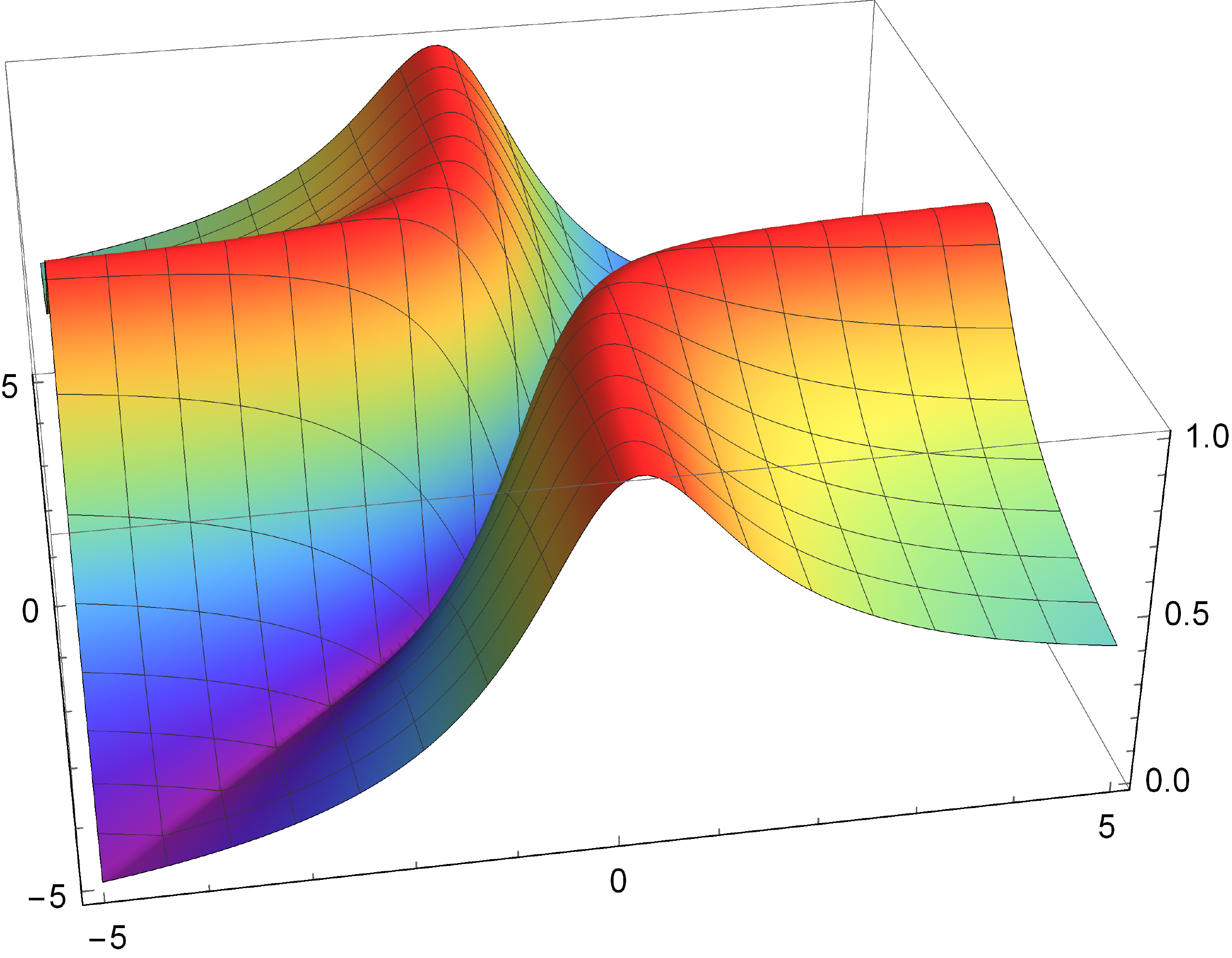}\caption{Chordal metric.}\label{cm}
	\end{figure}

	Let us check the hypotheses of Theorem~\ref{t_2}. We have that $d \in \mathcal{C}^2(\mathbb{R}^2\setminus \Delta, [0,\infty))$. Besides, positive definiteness and symmetry are clear. With respect to the cross partial derivative outside $\Delta$ we observe that \[ \partial_{12}d(x,y)=\operatorname{sgn}(x-y) \frac{2(x-y)}{\left(1+x^2\right)^{3 / 2}\left(1+y^2\right)^{3 / 2}}=\frac{2 \vert x-y \vert}{\left(1+x^2\right)^{3 / 2}\left(1+y^2\right)^{3 / 2}}.\] 
	which is greater than zero for every $(x,y)\notin\Delta$. Hence, we only need to check H4D. For any fixed $c$, we have that \[ \lim_{\lambda \to \infty} d(c,\lambda)=\lim_{\lambda \to \infty} \frac{2 (\lambda-c)}{\sqrt{1+c^2} \sqrt{1+\lambda^2}}=\frac{2}{\sqrt{1+c^2}} = \lim_{\lambda \to -\infty} \frac{2 (c-\lambda)}{\sqrt{1+c^2} \sqrt{1+\lambda^2}}=\lim_{\lambda \to -\infty} d(c,\lambda).\] 
\end{exa}

\begin{exa}[Generalized chordal metric]
Let $\alpha>0, \beta \geq 0, p \geq 1$. The function \[ d(x,y):=\frac{\left\lvert y-x\right\rvert}{\left(\alpha+\beta\left\lvert x\right\rvert^p\right)^{\frac{1}{p}} \cdot\left(\alpha+\beta\left\lvert y\right\rvert^p\right)^{\frac{1}{p}}},\]  is a distance.
\begin{figure}[h!]
	\centering
	\includegraphics[width=.5\linewidth]{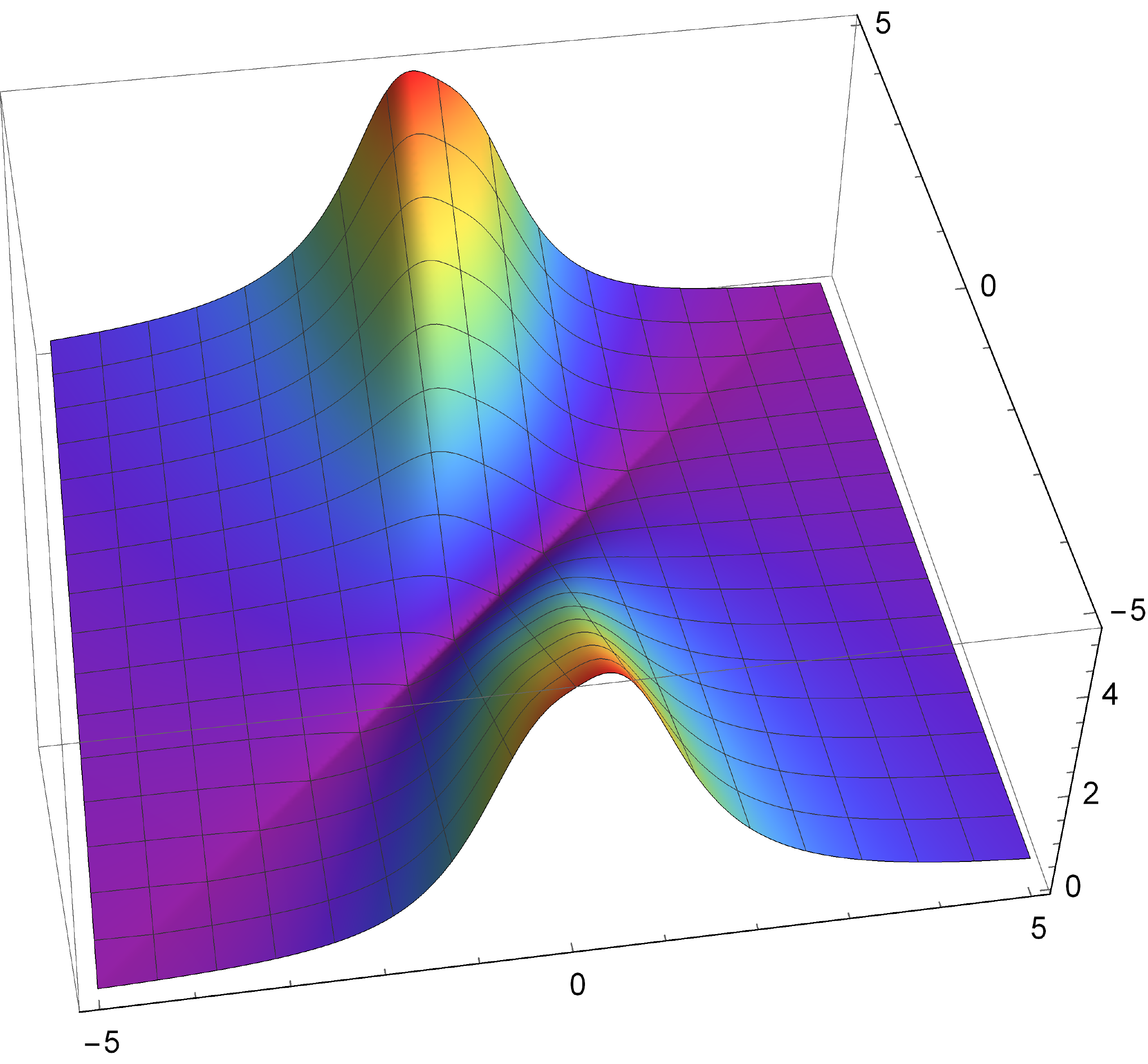}\caption{Generalized chordal metric for $\alpha=\beta=1$, $p=3$.}\label{gcm}
\end{figure}

In order to show that the previous function is a metric, we shall make some simplifications before checking that the cross partial derivative $\partial_{12}$ is non-negative. First note that \[ d(x,y):=\frac{1}{\alpha^{\frac{2}{p}}}\frac{\left\lvert y-x\right\rvert}{\left(1+\left(\beta/\alpha\right)\left\lvert x\right\rvert^p\right)^{\frac{1}{p}} \cdot\left(1+\left(\beta/\alpha\right)\left\lvert y\right\rvert^p\right)^{\frac{1}{p}}}.\]  With the idea of a change of variables, we rewrite the denominator \[ d(x,y):=\frac{1}{\alpha^{\frac{2}{p}}}\frac{\left\lvert y-x\right\rvert}{\left(1+\left(\beta^{\frac{1}{p}}\left\lvert x\right\rvert/\alpha^{\frac{1}{p}}\right)^p\right)^{\frac{1}{p}} \cdot \left(1+\left(\beta^{\frac{1}{p}}\left\lvert y\right\rvert/\alpha^{\frac{1}{p}}\right)^p\right)^{\frac{1}{p}}},\]  and the numerator \[ d(x,y):=\frac{1}{\alpha^{\frac{1}{p}}\beta^{\frac{1}{p}}}\frac{\left(\beta^{\frac{1}{p}}/\alpha^{\frac{1}{p}}\right)\left\lvert y-x\right\rvert}{\left(1+\left(\beta^{\frac{1}{p}}\left\lvert x\right\rvert/\alpha^{\frac{1}{p}}\right)^p\right)^{\frac{1}{p}} \cdot \left(1+\left(\beta^{\frac{1}{p}}\left\lvert y\right\rvert/\alpha^{\frac{1}{p}}\right)^p\right)^{\frac{1}{p}}}.\]  Thus, \[ d(g(x),g(y)):=\frac{1}{\alpha^{\frac{1}{p}}\beta^{\frac{1}{p}}}\frac{\left\lvert g(y)-g(x)\right\rvert}{\left(1+\left\lvert g(x)\right\rvert^p\right)^{\frac{1}{p}} \cdot\left(1+\left\lvert g(y)\right\rvert^p\right)^{\frac{1}{p}}},\]  where we have taken $g(s)=(\beta/\alpha)^{\frac{1}{p}}s$. Hence, if we define, $d_g=d \circ (g \times g)$, we need to check that \[ d_g(x,y):=\frac{\left\lvert y-x\right\rvert}{\left(1+\left\lvert x\right\rvert^p\right)^{\frac{1}{p}} \cdot\left(1+\left\lvert y\right\rvert^p\right)^{\frac{1}{p}}}\]  is a metric. Since the sufficiency theorems do not care about the smoothness of the metric candidate on the diagonal, the only apparent pending issue is the smoothness of its denominator. Hence, we shall use the map $h(x)=x^{2n+1}$ for $n$ such that $(2n+1)p > 2$ together with Theorem~\ref{t_5} in order to surpass this problem.

First, we have that $d_g \circ (h \times h)$ has the demanded regularity in Theorem~\ref{t_5}: $d_g$ is a class two function on $\mathbb{R}^2 \setminus \Delta$. Moreover, hypotheses H1 and H3 can be easily checked.

Besides, on the one hand, outside the diagonal ($\Delta$) or the points where $xy=0$ ($\Lambda$), the cross partial derivative is well-defined and it has the value \[ \partial_{12}d_g(x,y)=\frac{\operatorname{sgn}(y)\left\lvert y\right\rvert^{p-1}-\operatorname{sgn}(x)\left\lvert x\right\rvert^{p-1}}{\left(1+\left\lvert x\right\rvert^p\right)^{\frac{p+1}{p}} \cdot\left(1+\left\lvert y\right\rvert^p\right)^{\frac{p+1}{p}}},\]  which is greater than or equal to zero at any point. Therefore, H3 holds. On the other hand, H4 of Theorem~\ref{t_5} also holds for the case H4D, since \[ \lim_{\lambda \to - \infty} d_g(c,\lambda) = 1 = \lim_{\lambda \to +\infty} d_g(c,\lambda).\] 

\end{exa}
%
%
%

\bibliography{DR}
\bibliographystyle{spmpsciper}

\end{document}